%% file: main.tex
\documentclass[letterpaper, 10 pt, conference]{ieeeconf}  

\IEEEoverridecommandlockouts                              
\title{\LARGE \bf
Bandit Online Learning in Merely Coherent Games with Multi-Point Pseudo-Gradient Estimate
}

\usepackage{hyperref}
\usepackage{url}
\usepackage{amssymb}

\usepackage[dvipsnames]{xcolor}
\usepackage{mathrsfs}
\usepackage{bbm, dsfont}
\let\labelindent\relax 
\usepackage{amsmath, amsfonts, outlines, mathtools, outlines, cancel, enumitem, float, stackengine, graphicx, romannum}
\usepackage{subcaption}
\usepackage[english]{babel}
\usepackage[utf8]{inputenc}
\usepackage[T1]{fontenc}
\usepackage{newtxtext,newtxmath}
\usepackage{algorithm, algpseudocode}

\setenumerate[1]{label=(\roman*)}

\newtheorem{theorem}{Theorem}

\newtheorem{assumption}{Assumption}
\newtheorem{lemma}{Lemma}

\newtheorem{envdef}{Definition}

\DeclareMathSizes{10}{9}{6}{5}
\allowdisplaybreaks

\input{macros.tex}

\author{Yuanhanqing Huang$^{1}$ and Jianghai Hu$^{1}$
\thanks{This work was supported by the National Science Foundation under Grant No. 2014816 and No.2038410. }
\thanks{$^{1}$The authors are with the Elmore Family School of Electrical and Computer Engineering, Purdue University, West Lafayette, IN, 47907, USA 
        {\tt\small \{huan1282, jianghai\}@purdue.edu}}%
}

\newif\ifproceeding
\newif\ifarxiv

\proceedingfalse
\arxivtrue

\begin{document}

\maketitle
\thispagestyle{empty}
\pagestyle{empty}

\begin{abstract}
Non-cooperative games serve as a powerful framework for capturing the interactions among self-interested players and have broad applicability in modeling a wide range of practical scenarios, ranging from power management to drug delivery. 
Although most existing solution algorithms assume the availability of first-order information or full knowledge of the objectives and others' action profiles, there are situations where the only accessible information at players' disposal is the realized objective function values. 
In this paper, we devise a bandit online learning algorithm for merely coherent games that integrates the optimistic mirror descent scheme and multi-point pseudo-gradient estimates. 
We further demonstrate that the generated actual sequence of play can converge a.s. to a critical point if the sequences of query radius and sample size are chosen properly, without resorting to extra Tikhonov regularization terms or additional norm conditions. 
Finally, we illustrate the validity of the proposed algorithm via a Rock-Paper-Scissors game and a least square estimation game. 
\end{abstract}

\section{INTRODUCTION}

Recent years have witnessed considerably increasing interest in the analysis of multi-agent systems and large-scale networks, which find a wide range of applications such as thermal load management of autonomous buildings \cite{jiang2021game}, power management in sensor network \cite{campos2008game}, optimal drug delivery in the treatment of disease \cite{wu2012evolutionary}, control of environmental pollution \cite{du2015game}, etc. 
One primary objective in multi-agent systems is to devise local protocols for each agent, by following which, the resulting group behavior is optimal as measured by a certain system-level metric \cite{li2013designing}. 
With its origins in \cite{nash1950equilibrium}, game theory offers the theoretical tools to model and examine the strategic choices and associated outcomes of rational players who make decisions in a non-cooperative manner.  
In particular, in the Nash equilibrium problem (NEP), this group of players seeks to reach a stationary point known as Nash equilibrium (NE), where no rational player has any incentive to unilaterally deviate from it. 

In order to devise an algorithm for the NEP or its variants, it is crucial to have access to the first-order information, i.e., the partial gradient of the local objective function of each player, the evaluation of which nevertheless usually requires the action profile from all players. 
In view of this, in some studies \cite{mertikopoulos2019learning,yi2019operator,tatarenko2020geometric}, the availability of first-order oracles is taken as a given, whereas some other studies \cite{pavel2019distributed, bianchi2022fast, huang2022distributed} investigate network games where a communication network exists and players are willing to communicate with their trusted neighbors and keep local estimates of others' action profiles. 
Despite the notable progress discussed above, there are many real-world scenarios where players only have access to the observed objective values of selected actions, which makes the bandit/zeroth-order learning strategy a compelling choice. 
Our primary objective in this work is to develop an online learning algorithm for multi-player continuous games that possess mere coherence with 
bandit information. 

\textit{Related Work: }
\Tblue{There have been several recent notable contributions to the field of bandit learning in games. 
In their work \cite{bravo2018bandit}, Bravo et al. proposed a bandit version of mirror descent (MD), which guarantees a.s. convergence to an NE when the game is strictly monotone and achieves a convergence rate of $O(1/t^{1/3})$ for strongly monotone cases. 
Concerning the study of convergence rates in the realm of strongly monotone games or strongly variationally stable Nash equilibrium seeking, \cite{lin2021optimal, tatarenko2022rate, tatarenko2023convergence, drusvyatskiy2022improved} have succeeded in elevating the convergence rates from $O(1/t^{1/3})$ to $O(1/t^{1/2})$. 
Huang et al. \cite{huang2023zeroth} developed two bandit learning algorithms by integrating residual pseudo-gradient estimates into single-call extra-gradient schemes that ensure a.s. convergence to critical points of pseudo-monotone plus games. 
Moreover, in strongly pseudo-monotone plus games, by employing the proposed algorithms, the convergence rate is further elevated to  $O(1/t^{1-\epsilon})$. }

To extend the analysis beyond the realm of strictly monotone and pseudo-monotone plus games, Tatarenko et al. \cite{tatarenko2020bandit} utilized the single time-scale Tikhonov regularization and a doubly regularized approximate gradient descent strategy to develop an algorithm that converges to NEs in probability when the game is monotone and four decaying sequences are tuned properly. 
In a recent study \cite{gao2022bandit}, Gao et al. introduced an algorithm that integrates second-order learning dynamics and Tikhonov regularization and established the a.s. convergence of the sequence of play under the assumption that there exists at least one interior variationally stable state (VSS). 
Yet, the convergence is contingent on the norm condition that the $\ell_2$-norm of the state sequence should be greater than that of the VSS, which can be challenging to verify during the iterative process. 

In the literature of variational inequalities (VIs) and their stochastic versions (SVIs), Mertikopoulos et al. \cite{mertikopoulos2018optimistic} showed that the vanilla MD converges when the problem is strictly coherent, a relaxed variant of strict monotonicity, but fails to converge in merely coherent VIs. 
In contrast, the extra-gradient (EG) method is capable of achieving convergence to a solution in all coherent VIs, but it requires the exact operator values. 
In the presence of random noise in operator values, strict coherence is necessary to establish the convergence of the EG iteration. 
Similar convergence analysis is also reported in \cite{kannan2019optimal} for pseudo-monotone plus SVIs. 
To address the challenges posed by random noise, Iusem et al. \cite{iusem2017extragradient} developed an extra-gradient method for pseudo-monotone SVIs that incorporates an iterative variance reduction procedure and established both asymptotic convergence and non-asymptotic convergence rates for the proposed algorithm. 

\textit{Contributions: }
In this work, we develop a bandit online learning algorithm and establish the a.s. convergence of the generated sequence of play under the regularity condition that the game is merely coherent, which is broader and more general than the games investigated in \cite{tatarenko2020geometric, bravo2018bandit, lin2021optimal, huang2023zeroth}.  
The proposed algorithm leverages the optimistic mirror descent (OMD) \cite{azizian2021last, hsieh2019convergence}, a single-call extra-gradient scheme, as the backbone, which enables us to contend with the absence of strict coherence and reduces the query cost induced by the extra step. 
Alongside the OMD updates, the multi-point pseudo-gradient estimation is employed and the decaying rate of the variance of zeroth-order estimations can be controlled by properly tuning the query count per iteration. 
Furthermore, the validity of the proposed algorithm is verified through a Rock-Paper-Scissors game and a least square estimation game. 
\ifproceeding
All the proofs are included in \cite{huang2023bandit} due to the page limit. 
\fi

\textit{Basic Notations: } 
For a set of vectors $\{v_i\}_{i \in S}$, $[v_i]_{i \in S}$ or $[v_1; \cdots; v_{|S|}]$ denotes their vertical stack. 
For a vector $v$ and a positive integer $i$, $[v]_i$ denotes the $i$-th entry of $v$. 
We let $\norm{\cdot}$ denote the $\ell_2$-norm and $\langle, \rangle$ represent the canonical dot product. 
Let $\cl{\mathcal{S}}$ denote the closure of set $\mathcal{S}$, $\text{int}(\mathcal{S})$ the interior, and $\partial \mathcal{S}$ the boundary. 

\section{SETUP AND PRELIMINARIES}

\subsection{Game Formulation}
In a multi-player non-cooperative game $\mathcal{G}$ with the presence of $N$ players, indexed by $\mathcal{N} \coloneqq \{1, \ldots, N\}$, each player $i \in \mathcal{N}$ aims to optimize its own local objective $J^i$ by adjusting its action $x^i \in \mathcal{X}^i \subseteq \rset{n^i}{}$, which can be described as follows:
\begin{align}
\minimize_{x^i \in \mathcal{X}^i} J^i(x^i; x^{-i}),  
\end{align}
where $x^{-i} \coloneqq [x^j]_{j \in \mathcal{N}_{-i}}$ denotes the stack action of other players that parameterizes the objective $J^i$ with $\mathcal{N}_{-i} \coloneqq \mathcal{N}\backslash\{i\}$ and $x \coloneqq [x^j]_{j \in \mathcal{N}}$; 
$\mathcal{X}^i$ denotes the feasible set of player $i$, and for brevity, we let $\mathcal{X} \coloneqq \prod_{j \in \mathcal{N}} \mathcal{X}^j \subseteq \rset{n}{}$ represent the global strategy space and $\mathcal{X}^{-i} \coloneqq \prod_{j \in \mathcal{N}} \mathcal{X}^j \subseteq \rset{n^{-i}}{}$ with $n \coloneqq \sum_{j \in \mathcal{N}} n^j$ and $n^{-i} \coloneqq \sum_{j \in \mathcal{N}^{-i}} n^j$. 
Our blanket assumptions for the objective functions $J^i$'s and the local feasible sets $\mathcal{X}^i$'s will be as follows:

\begin{assumption}\label{asp:objt-set}
For each player $i$, the local objective function $J^i$ is continuously differentiable in $x$ over the global strategy space $\mathcal{X}$.
Moreover, its individual strategy space $\mathcal{X}^i$ is compact and convex, and has a non-empty interior. 
\end{assumption}

Given the smoothness posited in Assumption~\ref{asp:objt-set}, a single-valued operator that we will leverage extensively throughout is the pseudo-gradient operator $F: \rset{n}{} \to \rset{n}{}$. 
It is defined as the concatenation of all the partial gradient operators, i.e., 
\begin{align}
F: x \mapsto [\nabla_{x^i} J^i(x^i; x^{-i})]_{i \in \mathcal{N}}. 
\end{align}



Before proceeding, we remark that Assumption~\ref{asp:objt-set} implicitly implies that $F$ is Lipschitz continuous on $\mathcal{X}$ with some constant $L$, i.e., for any $x$ and $x^\prime \in \mathcal{X}$, we have 
\begin{align}
\norm{F(x) - F(x^\prime)} \leq L\norm{x - x^\prime}. 
\end{align}

As for the solution concept, we focus on critical points (CPs) \cite[Sec.~2.2]{mertikopoulos2022learning}, a more relaxed solution concept than Nash equilibria (NEs), whose definition is given as follows. 
\begin{envdef}\label{def:cps} (Critical Points)
A decision profile $x_* \in \mathcal{X}$ is a critical point of the game $\mathcal{G}$ if it is a solution to the associated (Stampacchia) variational inequality (VI), i.e., 
\begin{align}\label{eq:cps}
\langle F(x_*), x - x_*\rangle \geq 0, \; \forall x \in \mathcal{X}. 
\end{align}
\end{envdef}
We postulate that the games discussed in this work admit at least one critical point inside $\mathcal{X}$. 
A well-known result is that CPs coincide with NEs when $J^i$ is convex and continuously differentiable in $x^i$ for all $i$ \cite[Sec.~1.4.2]{facchinei2003finite}. 

In this work, our aim is to propose a new algorithm that is applicable to a broader class of games as compared to strictly monotone games and pseudo-monotone plus games. 
Moreover, we intend to further relax pseudo-monotonicity assumptions that are usually imposed upon the structure of the game to the ones merely upon equilibria. 
\begin{assumption}\label{asp:vs} (Mere Coherence)
The game $\mathcal{G}$ is merely coherent if every critical point (CP) $x_*$ of $\mathcal{G}$ is merely variationally stable, i.e., $\langle F(x), x - x_*\rangle \geq 0$ for all $x \in \mathcal{X}$. 
\end{assumption}

Before we proceed, it is pertinent to make a few comments. 
Our analysis primarily lies within Euclidean space; however, we recognize the potential for extending its applicability to finite-dimensional Hilbert spaces.
In addition, we employ mere coherence rather than pseudo-monotonicity as the standing assumption, as the former one is less restrictive. 
Recall that an operator $F$ is pseudo-monotone if for all $x, y \in \mathcal{X}$, $\langle F(y), x - y\rangle \geq 0 \implies \langle F(x), x - y\rangle \geq 0$. 
Nonetheless, the latter is generally the more readily verifiable assumption in practical applications, since it does not needs the CPs $x_*$'s to be known a priori.


\subsection{Optimistic Mirror Descent}\label{subsec:omd}

In this subsection, we shall provide a brief overview of the optimistic mirror descent algorithm, as well as related concepts and results. 
As an extension of the Euclidean projection, the mirror map $\nabla \psi^*: \rset{}{} \to \rset{}{}$ is defined as: 
\begin{align}
\nabla \psi^*(z) = \argmax_{x \in \mathcal{X}} \{\langle z, x\rangle - \psi(x)\}, 
\end{align}
where $\psi:\dom\psi \to \rset{}{}$ is a so-called distance-generating function (DGF) with $\dom\psi$ denoting a convex and open set where $\psi$ is well-defined. 
The DGF fulfills the following conditions \cite[Sect.~4.1]{bubeck2014theory}:
$(\romannum{1})$ $\psi$ is differentiable and $\Tilde{\mu}$-strongly convex for some $\Tilde{\mu} > 0$; 
$(\romannum{2})$ $\nabla \psi(\dom \psi) = \rset{n}{}$; 
$(\romannum{3})$ $\cl{\dom \psi} \supseteq \mathcal{X}$ and $\lim_{x \to \partial (\dom \psi)}$ $\norm{\nabla \psi(x)}_* = +\infty $. 
The definition of DGF $\psi$ allows us to introduce a pseudo-distance called the Bregman divergence, which is defined as:
\begin{align}
D(p,x) = \psi(p) - \psi(x) - \langle \nabla \psi(x), p - x\rangle, \forall p, x \in \dom \psi.
\end{align}
To let $D(p, \cdot)$ represent a certain distance measure to $p$ and use this measure to define a neighborhood of $p$, we make the following assumption. 
\begin{assumption}\label{asp:recip}
(Bregman Reciprocity) The chosen DGF $\psi$ satisfies that if the sequence $(x_k)_{k \in \nset{}{+}}$ converges to some point $p$, i.e., $\norm{x_k - p} \to 0$, then $D(p, x_k) \to 0$.
\end{assumption}
Then, the Bregman divergence generates the prox-mapping $P_{x, \mathcal{X}}: \mathcal{H} \to \dom{\psi} \cap \mathcal{X}$ for some fixed $x \in \mathcal{X} \cap \dom{\psi}$ that plays a critical role in mirror descent and its variants:
\begin{align}
P_{x, \mathcal{X}}(y) = \argmin_{x^\prime \in \mathcal{X}}\{\langle y, x - x^\prime\rangle + D(x^\prime, x)\}.
\end{align}

With all these in hand, the optimistic mirror descent (OMD) \cite{azizian2021last, hsieh2019convergence} can be expressed as below: 
\begin{align}\label{eq:omd}
\begin{split}
X_{k+1/2} &= P_{X_k, \mathcal{X}}(-\tau_k F(X_{k-1/2})) \\
X_{k+1} &= P_{X_k, \mathcal{X}}(-\tau_k F(X_{k+1/2})),
\end{split}
\end{align}
\Tblue{where $(\tau_k)_{k \in \nset{}{+}}$ denotes a proper sequence of step sizes. }
The update consists of the following two steps. 
Given the base state $X_k$ at step $k$, in the look-forward step, the leading state $X_{k+1/2}$ is procured by updating $X_k$ with the proxy $F(X_{k-1/2})$ queried at $X_{k-1/2}$ rather than the exact pseudo-gradient $F(X_{k})$ queried at $X_k$ to reduce the oracle call per iteration. 
This step is essential in anticipating the landscape of $F$ and facilitating the convergence when $F$ is merely monotone, i.e., $\langle F(x) - F(y), x - y \rangle \geq 0$, for all $x$ and $y$ feasible. 
In the state-updating step, the base state $X_k$ is revised to $X_{k+1}$ following the pseudo-gradient information $F(X_{k+1/2})$. 
The OMD falls into the single-call category, distinguishing itself from the conventional extra gradient algorithm \cite{iusem2017extragradient} by exclusively utilizing the first-order information at $X_{k+1/2}$, without requiring information from both $X_{k}$ and $X_{k+1/2}$.


\section{MULTI-POINT PSEUDO-GRADIENT ESTIMATION}

In this paper, we examine the scenario where the first-order information at the leading state, i.e., $F(X_{k+1/2})$ is not readily available, and players need to estimate them based on the realized objective function values. 
A prevalent technique in the literature of first-order information estimation methods is the simultaneous perturbation stochastic approximation (SPSA) approach \cite{bravo2018bandit}. 
For each $i \in \mathcal{N}$, let $\mathbb{B}_i, \mathbb{S}_i \subseteq \rset{n^i}{}$ denote the unit ball and the unit sphere centered at the origin. 
At each iteration $k$, before implementing the SPSA estimate, we initially undertake the following perturbation step: 
\begin{align}\label{eq:perb}
\begin{split}
& \hat{X}^i_{k+1/2} = (1 - \frac{\delta_k}{r^i})X^i_{k+1/2}+\frac{\delta_k}{r^i}(p^i + r^iu^i_k) = \bar{X}^i_{k+1/2} + \delta_k u^i_k,
\end{split}
\end{align}
where 
$u^i_k$ is randomly sampled from $\mathbb{S}_{i} \subseteq \rset{n^i}{}$ and we define $u_k \coloneqq [u^i_k]_{i \in \mathcal{N}}$; 
$\delta_k$ represents the random query radius at iteration $k$;
$\mathbb{B}(p^i, r^i) \subseteq \mathcal{X}^i$ is an arbitrary fixed Euclidean ball within the feasible set $\mathcal{X}^i$ that centers at $p^i$ with radius $r^i$; 
$\bar{X}^i_{k+1/2} \coloneqq (1 - \delta_k/r^i)X^i_{k+1/2} + (\delta_k/r^i) p^i$. 
Denote $\bar{X}_{k+1/2} \coloneqq [\bar{X}^i_{k+1/2}]_{i \in \mathcal{N}}$. 
In the merit of the feasibility adjustment in \eqref{eq:perb}, the action to be taken will sit within the feasible set, i.e., $\hat{X}^i_{k+1/2} \in \mathcal{X}^i$ and $\hat{X}_{k+1/2} \coloneqq [\hat{X}^i_{k+1/2}]_{i \in \mathcal{N}} \in \mathcal{X}$. 
With this in hand, the SPSA estimation can be expressed as $\frac{n^i}{\delta_k} J^i(\hat{X}_{k+1/2})u^i_k$. 
Nevertheless, as previously noted in \cite{bravo2018bandit}, the SPSA approach incurs a larger estimation variance with a decrease in query radius aimed at improving estimation accuracy, which results in conservative choices of updating step sizes $\tau_k$ and significant degradation of the convergence rate. 
To resolve this conundrum, there has been increased consideration given to schemes such as two-point estimation and residual estimation to keep the variance bounded. 
On account of this, we consider the multi-point pseudo-gradient estimation (MPG) scheme, the counterparts of which in the field of optimization can be found in \cite{duchi2015optimal}.
At every iteration $k$, each player $i$ executes the perturbation step in \eqref{eq:perb} $(T_k + 1)$ times in an independent manner, takes the action $\hat{X}^i_{k+1/2,t}$, and observes the associated realized objective function values $J^i(\hat{X}_{k+1/2,t})$, where the variable $t \in \nset{}{}$ is an index of the multiple samples taken per iteration.
The multi-point pseudo-gradient estimate can be formulated as below:
\begin{align}\label{eq:mpg}
\tag{MPG}
G^i_k \coloneqq \frac{n^i}{\delta_kT_k}\sum_{t=1}^{T_k}\big(J^i(\hat{X}_{k+1/2,t}) - J^i(\hat{X}_{k+1/2,0})\big)u^i_{k,t}, 
\end{align}
where $(u^i_{k,t})_{t=0, \ldots, T_k}$ are i.i.d. random variables uniformly distributed over $\mathbb{S}_i$; 
the action taken by player $i$ is given by 
$\hat{X}^i_{k+1/2,t} \coloneqq (1 - \frac{\delta_k}{r^i})X^i_{k+1/2}+\frac{\delta_k}{r^i}(p^i + r^iu^i_{k,t}) = \bar{X}^i_{k+1/2} + \delta_k u^i_{k,t}$; $\hat{X}_{k+1/2,t} \coloneqq [\hat{X}^i_{k+1/2,t}]_{i \in \mathcal{N}}$. 
To simplify the presentation, we will henceforth use $\hat{J}^i_{k,t}$ to represent the realized objective value $J^i(\hat{X}_{k+1/2,t})$ for the $t$-th sample at iteration $k$. 
Prior to delving into the properties of \ref{eq:mpg}, we first outline the probability setup to streamline our later discussion. 
Let $(\Omega, \mathcal{F}, \mathcal{P})$ denote the underlying probability space. 
The filtration $(\mathcal{F}_k)_{k \in\nset{}{+}}$ is constructed as $\mathcal{F}_k \coloneqq \sigma\big\{X_0, \{u_{1,t}\}_{t=0}^{T_1}, \ldots, \{u_{k-1,t}\}_{t=0}^{T_{k-1}}\big\}$, which captures the update that results in $X_k$, i.e., the entire information up to and including iteration $k-1$. . 
Then to characterize \ref{eq:mpg}, we start by considering the following decomposition of it:
\begin{align*}
G^i_k =& \nabla_{x^i} J^i(X_{k+1/2}) + \big(G^i_k - \expt{}{G^i_k \mid \mathcal{F}_k}\big) \\
& + \big(\expt{}{G^i_k \mid \mathcal{F}_k} - \nabla_{x^i} J^i(X_{k+1/2})\big). 
\end{align*}
For brevity, we let $B^i_k \coloneqq \expt{}{G^i_k \mid \mathcal{F}_k} - \nabla_{x^i} J^i(X_{k+1/2})$ represent the systematic error and $V^i_k \coloneqq G^i_k - \expt{}{G^i_k \mid \mathcal{F}_k}$ the stochastic error. 
To facilitate later analysis, for each $J^i$, we introduce the $\delta$-smoothed objective function $\Tilde{J}^i_{\delta}$: 
\begin{align}
\Tilde{J}^i_\delta(x^i; x^{-i}) \coloneqq \frac{1}{\mathbb{V}^i_{\delta}}\int_{\delta \mathbb{S}_{-i}}\int_{\delta \mathbb{B}_i} J^i(x^i + \Tilde{\tau}^i; x^{-i} + \tau^{-i})d\Tilde{\tau}^id\tau^{-i}, 
\end{align}
where $\mathbb{S}_{-i} \coloneqq \prod_{j \in \mathcal{N}^{-i}} \mathbb{S}_j \subseteq \rset{n^{-i}}{}$; 
$\mathbb{V}^i_{\delta} \coloneqq \vol{\delta \mathbb{B}_i} \cdot \vol{\delta \mathbb{S}_{-i}}$.  
The lemmas presented below provide an examination of the properties of $B^i_k$ and $V^i_k$, which will be later employed in the proof of the main theorem. 

\begin{lemma}\label{le:bias}
Suppose that Assumption~\ref{asp:objt-set} holds. 
Then at each iteration $k$, the conditional expectation satisfies $\expt{}{G^i_k \mid \mathcal{F}_k} = \nabla_{x^i} \Tilde{J}^i_{\delta_k}(\bar{X}_{k+1/2})$ a.s. for every $i \in \mathcal{N}$. 
Moreover the systematic error $B_k \coloneqq [B^i_k]_{i \in \mathcal{N}}$ possesses a decaying upper bound $\norm{B_k} \leq \alpha_B \delta_k$ for some positive constant $\alpha_B$.
\end{lemma}
\begin{proof}
\ifarxiv
See Appendix~\ref{pf:bias}. 
\fi
\ifproceeding
See \cite[Appendix~A]{huang2023bandit}.
\fi
\end{proof}

In contrast to the single-point or two-point estimates, the advantage of utilizing \ref{eq:mpg} is primarily demonstrated in the following lemma, which measures the decaying rate of the stochastic error w.r.t. the number of samples. 
\begin{lemma}\label{le:variance}
Suppose that Assumption~\ref{asp:objt-set} holds. 
Then at each iteration $k$, the squared norm of $V_k \coloneqq [V^i_k]_{i \in \mathcal{N}}$ satisfies $\expt{}{\norm{V_k}^2 \mid \mathcal{F}_k} \leq \alpha_V / T_k$ for some positive constant $\alpha_V$. 
\end{lemma}
\begin{proof}
\ifarxiv
See Appendix~\ref{pf:variance}. 
\fi
\ifproceeding
See \cite[Appendix~B]{huang2023bandit}.
\fi
\end{proof}



\section{A VARIANCE-REDUCTION LEARNING ALGORITHM AND CONVERGENCE ANALYSIS}

In view of the convergence properties of OMD introduced in Sec.~\ref{subsec:omd}, we design a zeroth-order algorithm for merely monotone games by incorporating \ref{eq:mpg} into OMD, the precision of which can be controlled by adjusting the sample size per iteration. 
Each player of the group possesses their own local $\Tilde{\mu}^i$-strongly convex DGF, denoted by $\psi^i$. 
Additionally, the function $\psi(x) \coloneqq \sum_{i \in \mathcal{N}} \psi^i(x^i)$ with $x \coloneqq [x^i]_{i \in \mathcal{N}}$ represents the group DGF, which is $\Tilde{\mu}$-strongly convex. 
The proposed approach is outlined in Algorithm~\ref{alg:vr-bdt-omd}. 

\begin{algorithm}
\caption{Zeroth-Order Variance-Reduced Learning of CPs Based on Optimistic Mirror Descent (Player $i$)}\label{alg:vr-bdt-omd}
\begin{algorithmic}[1]
\State \textbf{Initialize:} $X^i_0 = X^i_{1/2} = X^i_1 \in \mathcal{X}^i \cap \dom{\psi^i}$ arbitrarily; 
$G^i_{0} = \bzero_{n^i}$; 
$p^i, r^i$ to be the center and radius of an arbitrary ball within the set $\mathcal{X}^i$
\Procedure{At the $k$-th iteration ($k \in \nset{}{+}$)}{}
\State $X^i_{k+1/2} \leftarrow P_{X^i_k, \mathcal{X}^i}(-\tau G^i_{k-1})$ 
\For{$t = 0, \ldots, T_k$}
    \State Randomly sample the direction $u^i_{k,t}$ from $\mathbb{S}_i$
    \State $\hat{X}^{i}_{k+1/2,t} \leftarrow (1 - \frac{\delta_k}{r^i})X^{i}_{k+1/2,t} + \frac{\delta_k}{r^i}(p^i + r^i u^{i}_{k,t})$ 
    \State Take action $\hat{X}^{i}_{k+1/2,t}$
    \State Observe the realized objective function value $\hat{J}^{i}_{k,t} \coloneqq J^i(\hat{X}^{i}_{k+1/2,t}; \hat{X}^{-i}_{k+1/2,t})$
\EndFor
\State $G^i_k \leftarrow \frac{n^i}{\delta_kT_k}\sum_{t=1}^{T_k}(\hat{J}^{i}_{k,t} - \hat{J}^{i}_{k,0})u^{i}_{k,t} = \frac{1}{T_k}\sum_{t=1}^{T_k}G^{i}_{k,t}$
\State $X^i_{k+1} \leftarrow P_{X^i_k, \mathcal{X}^i}(-\tau G^i_{k})$
\EndProcedure
\State \textbf{Return:} $\{\hat{X}^i_{k+1/2}\}_{i \in \playerN}$
\end{algorithmic}
\end{algorithm}

The Robbins-Siegmund (R-S) theorem serves as a heavy-lifting tool in the field of stochastic optimization to examine the convergence of sequences.
Its formal statement is presented as follows. 

\begin{lemma}(\cite[Thm.~1]{robbins1971convergence})
Let $(\Omega, \mathcal{F}, \mathcal{P})$ be a probability space and $(\mathcal{F}_k)_{k}$ a filtration of $\mathcal{F}$. 
For each $k = 1, 2, \ldots$, $Z_k$, $\beta_k$, $\xi_k$, and $\zeta_k$ are non-negative $\mathcal{F}_k$-measurable random variables that satisfy $\expt{}{Z_{k+1} \mid \mathcal{F}_k} \leq (1 + \beta_k)Z_{k} + \xi_k - \zeta_k$. 
If $\sum_{k \in \nset{}{+}} \beta_k < \infty$ a.s. and $\sum_{k \in \nset{}{+}} \xi_k < \infty$ a.s., then $\lim_{k \to \infty} Z_k$ exists and is finite a.s. and $\sum_{k \in \nset{}{+}} \zeta_k < \infty$ a.s.
\end{lemma}

To employ the theorem, it is necessary to guarantee that $\sum_{k \in \nset{}{+}}\xi_k$ is finite a.s. 
Recall from Lemma~\ref{le:variance}, in the variance reduction scenario, the decaying upper bound is constructed for $\expt{}{\norm{V_k}^2 \mid \mathcal{F}_k}$ rather than the random variable $\norm{V_k}^2$. 
In the meantime, unlike the typical extra-gradient method, OMD leverages the pseudo-gradient $F(X_{k-1/2})$ from the last iteration when updating to the leading state $X_{k+1/2}$. 
This approximation brings the stochastic error $\norm{V_{k-1}}^2$ into the recurrent inequality which, due to the absence of the averaging effect, does not possess a decaying upper bound and prevents us from applying the R-S theorem. 
Motivated by the consideration above, our next step will be establishing a variant of the R-S theorem by relaxing the condition imposed upon the sequence $(\xi_k)_{k \in \nset{}{+}}$.  

\begin{theorem}\label{thm:ext-rs}
Let $(\Omega, \mathcal{F}, \mathcal{P})$ be a probability space and $(\mathcal{F}_k)_{k}$ a filtration of $\mathcal{F}$. 
For each $k = 1, 2, \ldots$, $Z_k$, $\xi_k$, and $\zeta_k$ are non-negative $\mathcal{F}_k$-measurable random variables that satisfy $\expt{}{Z_{k+1} \mid \mathcal{F}_k} \leq Z_{k} + \xi_k - \zeta_k$ with $\expt{}{Z_1} < \infty$. 
If $\sum_{k \in \nset{}{+}} \expt{}{\xi_k} < \infty$, then $Z_k$ converges a.s. to some random variable $Z_{\infty}$ with $\expt{}{Z_{\infty}} < \infty$ and $\sum_{k \in \nset{}{+}} \zeta_k < \infty$ a.s.
\end{theorem}
\begin{proof}
\ifarxiv
See Appendix~\ref{pf:ext-rs}. 
\fi
\ifproceeding
See \cite[Appendix~C]{huang2023bandit}.
\fi
\end{proof}

With this conclusion available, we can establish the following results about the convergence of Algorithm~\ref{alg:vr-bdt-omd} and the sufficient conditions to guarantee it. 

\begin{theorem}\label{thm:convg}
Consider a multi-player game $\mathcal{G}$. 
Suppose that Assumptions~\ref{asp:objt-set} to \ref{asp:recip} hold.
In addition, the sequence of query radius $(\delta_k)_{k \in \nset{}{+}}$ and the sequence of the reciprocal of sample size $(1/T_k)_{k \in \nset{}{+}}$ are monotonically decreasing and satisfy
\begin{align}
\sum_{k \in \nset{}{+}} \delta_k < \infty, \;\sum_{k \in \nset{}{+}}1/T_k < \infty.
\end{align}
The step size $\tau$ satisfies $(\tau L /\Tilde{\mu})^2 \leq 1/12$. 
Then the base state $(X_{k})_{k \in \nset{}{+}}$ as well as the leading state $(X_{k+1/2})_{k \in \nset{}{+}}$ converge a.s. to a CP $x_*$ of $\mathcal{G}$. 
Moreover, the actual sequence of play also satisfy $\lim_{k \to \infty} \hat{X}_{k+1/2,t} = x_*$ a.s., for arbitrary sample $t$. 
\end{theorem}
\begin{proof}
\ifarxiv
See Appendix~\ref{pf:convg}. 
\fi
\ifproceeding
See \cite[Appendix~D]{huang2023bandit}.
\fi
\end{proof}


\section{NUMERICAL EXPERIMENTS}
\subsection{The Rock-Paper-Scissors (RPS) Game}
Consider the zero-sum rock-paper-scissors game between two players. 
The payoff matrices $A^a$ and $A^b$ of player $a$ and $b$ are set respectively as
\begin{align*}
A^a \coloneqq \begin{bmatrix} 0 & -1 & 1 \\ 1 & 0 & -1 \\ -1 & 1 & 0 \end{bmatrix}
\;\text{and}\;
A^b \coloneqq \begin{bmatrix} 0 & 1 & -1 \\ -1 & 0 & 1 \\ 1 & -1 & 0 \end{bmatrix} = -A^a, 
\end{align*}
which further give rise to the objective functions: $J^a(x^a; x^b) = -(x^a)^TA^ax^b$ and $J^b(x^b; x^a) = -(x^a)^TA^bx^b$. 
The associated strategy spaces are the probability simplices, i.e., $\mathcal{X}^a = \mathcal{X}^b \coloneqq \{x \in \rset{3}{} \mid 0 \leq x \leq 1, \bone^T x = 1\}$. 
The RPS game is merely monotone and admits a unique CP/NE at $[1/3; 1/3; 1/3]$ for both players. 
To fulfill the requirement about the non-empty interior in Assumption~\ref{asp:objt-set}, taking player $a$ as an example, we can employ a simple coordinate transformation $\varphi: \rset{2}{} \to \rset{3}{}$ with $\varphi: [y_1; y_2] \mapsto [y_1; y_2; 1 - y_1 - y_2] = [x_1; x_2; x_3]$ and $\varphi^{-1}(\mathcal{X}^a) = \Tilde{\mathcal{X}}^a = \{y \in \rset{2}{} \mid 0 \leq y \leq 1, \bone^T y \leq 1\}$. 
Then \ref{eq:mpg} is applied to obtain a pseudo-gradient estimate $\Tilde{G}_k \in \rset{2}{}$, and we use another map $\phi: \rset{2}{} \to \rset{3}{}$ to pull the pseudo-gradient from $y$-coordinate system back to $x$-coordinate system. 
The map $\phi$ is defined as $\phi: [\Tilde{g}_1; \Tilde{g}_2] \mapsto [2/3 \Tilde{g}_1 - 1/3\Tilde{g}_2; -1/3\Tilde{g}_1 + 2/3\Tilde{g}_2; -1/3\Tilde{g}_1-1/3\Tilde{g}_2]$, which is derived from the observation that  
\begin{align*}
    \Tilde{g}_i = \sum_{j=1,2,3}\frac{\partial J^a}{\partial x_j}\cdot \frac{\partial x_j}{\partial y_i} = \sum_{j=1,2,3} g_j\cdot \frac{\partial x_j}{\partial y_i} \;\text{and}\;\sum_{j=1,2,3}g_j = 0. 
\end{align*}
A similar procedure can be applied to player $b$ to guarantee the fulfillment of the assumption. 

In the numerical simulation, we choose $\tau = 0.1$, the decaying query radius $\delta_k = 0.1 \times (k + 20)^{-1.1}$, and the increasing number of queries per iteration $T_k = \ceil{10^{-3} \times k^{1.1} + 20}$. 
Since the negative entropy $h(x) = \sum_{i=1,2,3} [x]_i\log[x]_i$ is $1$-strongly convex in $\norm{\cdot}$ and satisfies all the requirements discussed in Sec.~\ref{subsec:omd}, it can be chosen as a DGF for player $a$ and $b$. 
The simulation results are illustrated in Fig.~\ref{fig:rps}, with Fig.~\ref{fig:rps} (a) and (b) visualizing the actual sequences of play of player $a$ and $b$.
To compare with \cite{gao2022bandit} (MD2-rb), Fig.~\ref{fig:rps} (c) and (d) illustrate the relative distance $\norm{\hat{X}_{k+1/2} - x_*}/\norm{x_*}$ to the CP/NE $x_*$, where the $x$-axis denotes the sample count and iteration index, respectively. 
The selection of parameters for \cite{gao2022bandit} (MD2-rb) adheres to the specifications provided in its corresponding section. 
As depicted in the figure, \cite{gao2022bandit} (MD2-rb) displays a faster decline in the early iterations, whereas Algorithm~\ref{alg:vr-bdt-omd} achieves a superior convergence rate as the progress advances. 
\begin{figure}
    \centering
    \includegraphics[width=0.48\textwidth]{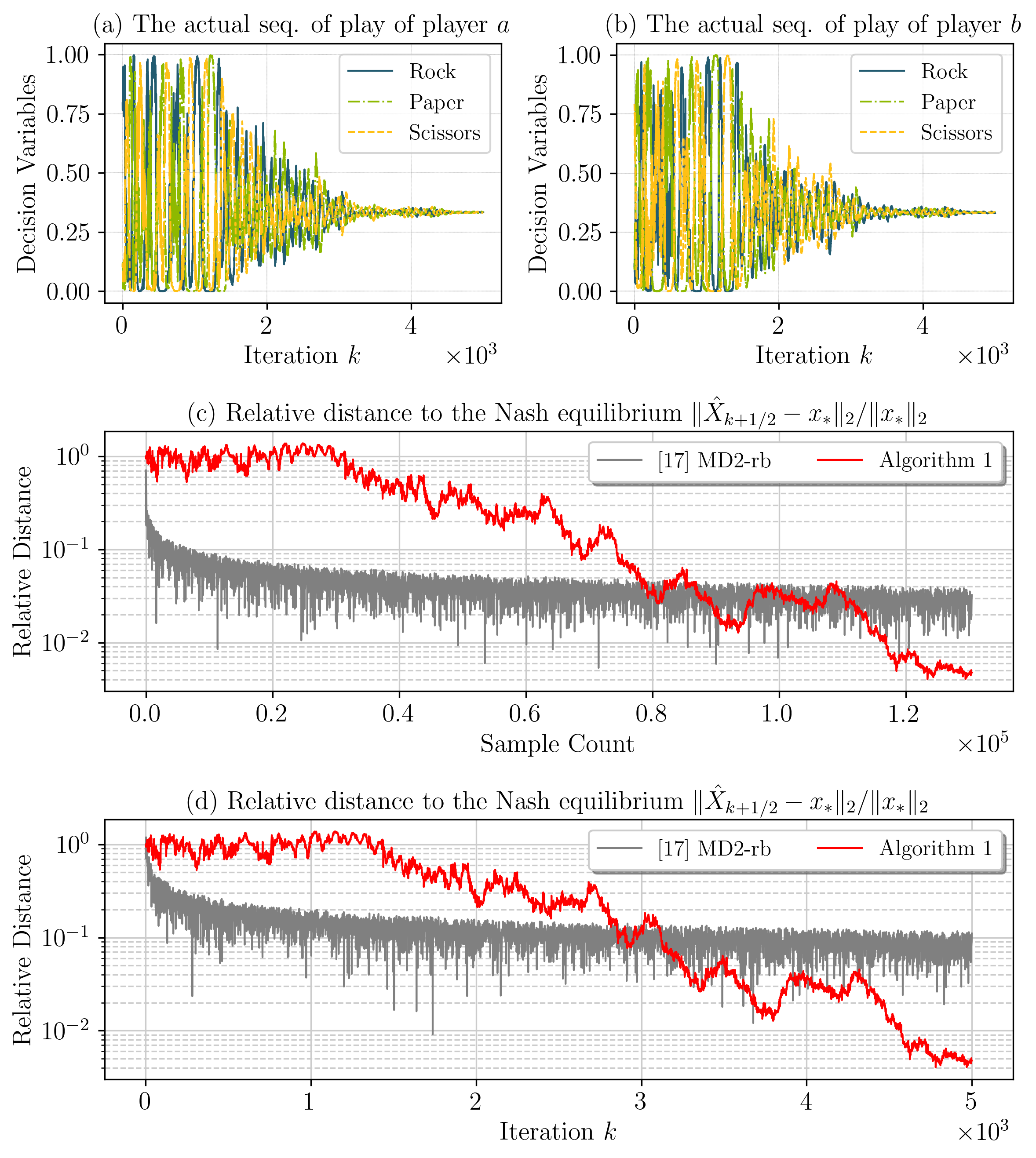}
    \caption{Performance of Algorithm~\ref{alg:vr-bdt-omd} in the RPS Game}
    \label{fig:rps}
\end{figure}

\subsection{Least Square Estimation in Linear Models}
In this numerical experiment, we convert the linear regression to a zero-sum bilinear game between two players \cite[Sec.~VI]{gao2020continuous}. 
Given a set of data samples $\{(z_j, y_j)\}_{j=1}^{M}$ where $z_j \in \rset{N}{}$ and $y_j \in \rset{}{}$ represent the input feature vector and the output scalar, respectively. 
In addition, $y_j = w_0 + w^Tz_j + \xi_j$, with $\Tilde{w} \coloneqq [w_0; w] \in [-\bar{w}, \bar{w}]^{N+1} \subseteq \rset{N+1}{}$ denoting the parameters to be determined and $\xi_j$ some random noise. 
Here, the region $[-\bar{w}, \bar{w}]^{N+1}$ with $\bar{w} \in \rset{+}{}$ is enforced to ensure the strategy space is bounded. 
For brevity, denote $\Tilde{z}_j \coloneqq [1; z_j]$, $\Tilde{Z} \coloneqq [\Tilde{z}_1, \ldots, \Tilde{z}_M]$ and $y = [y_1; \cdots; y_M]$. 
We can then formulate this least square estimation problem as: 
\begin{align}\label{eq:lse-obj}
\underset{\Tilde{w} \in [-\Tilde{w}, \Tilde{w}]^{N+1}}{\minimize} \; \frac{1}{2}\norm{\Tilde{Z}^T\Tilde{w} - y}^2_2. 
\end{align}
To convert it into a two-player game, we leverage an auxiliary variable $\lambda \in \rset{M}{}$ and the fact that $\frac{1}{2}\norm{\Tilde{Z}^T\Tilde{w} - y}^2_2 = \max_{\lambda \in \rset{M}{}}\lambda^T(\Tilde{Z}^T\Tilde{w} - y) - \frac{1}{2}\norm{\lambda}^2_2 = \max_{\lambda \in \rset{M}{}} J(\Tilde{w}, \lambda)$. 
Taking the boundedness of $\Tilde{w}$ into account, it can be asserted that there exists a bounded set $[-\bar{\lambda}, \bar{\lambda}]^M$ such that the solution $\lambda$ to the maximization problem above satisfies $\lambda \in [-\bar{\lambda}, \bar{\lambda}]^M$. 
As such, let $J^1(x^1; x^2) = J(x^1, x^2)$ and $J^2(x^2; x^1) = -J(x^1, x^2)$, and this game can be formulated as follows: 
\begin{align*}
\text{Player 1: }\underset{-\Bar{w} \leq x^1 \leq \Bar{w}}{\minimize}\; J^1(x^1; x^2), \; \text{Player 2: }\underset{-\Bar{\lambda} \leq x^2 \leq \Bar{\lambda}}{\minimize}\; J^2(x^2; x^1). 
\end{align*}
For the verification of the remaining assumptions, showing the uniqueness of the CP, and other detailed discussions, we refer the interested reader to \cite[Sec.~V-B]{huang2023zeroth}\cite[Sec.~VI]{gao2020continuous}. 

When implementing the experiments, we choose $N = 5$, $M = 20$, and $\bar{w} = \bar{\lambda} = 5$. 
Then random noise $\xi_i$ is uniformly distributed over the interval $[-0.6, 0.6]$. 
We compare different sets of the sequences of query radius $\delta_k$ and query samples per iteration $T_k$. 
In Fig.~\ref{fig:lse} (a), the original curve to fit, the noisy data samples used, the optimal solution that can be procured from the existing data, and one OMD solution generated by Algorithm~\ref{alg:vr-bdt-omd} are illustrated. 
Comparing the results with different choices of $\delta_k$, we note that for this problem when $\delta_k$ decays comparable to or faster than $O(k^{-1.1})$, further increasing the decaying rate contributes little to speed up the convergence rate of the sequence. 
As for the influence of different $T_k$, when $T_k$ is a small constant, the generated sequences will diverge; 
when $T_k = 10$ increases to some sufficiently large constant $T_k=15$, the associated sequences demonstrate the trend of convergence to some $\epsilon$-neighborhood of the CP; 
when $T_k$ decays no slower than $O(k^{-1.1})$, as reflected in Fig.~\ref{fig:lse} (c) and (e), the fluctuations of the relative step sizes are mitigated; yet little difference can be observed regarding the relative distance to the CP, as shown in Fig.~\ref{fig:lse} (b) and (d). 

\begin{figure}
    \centering
    \includegraphics[width=0.48\textwidth]{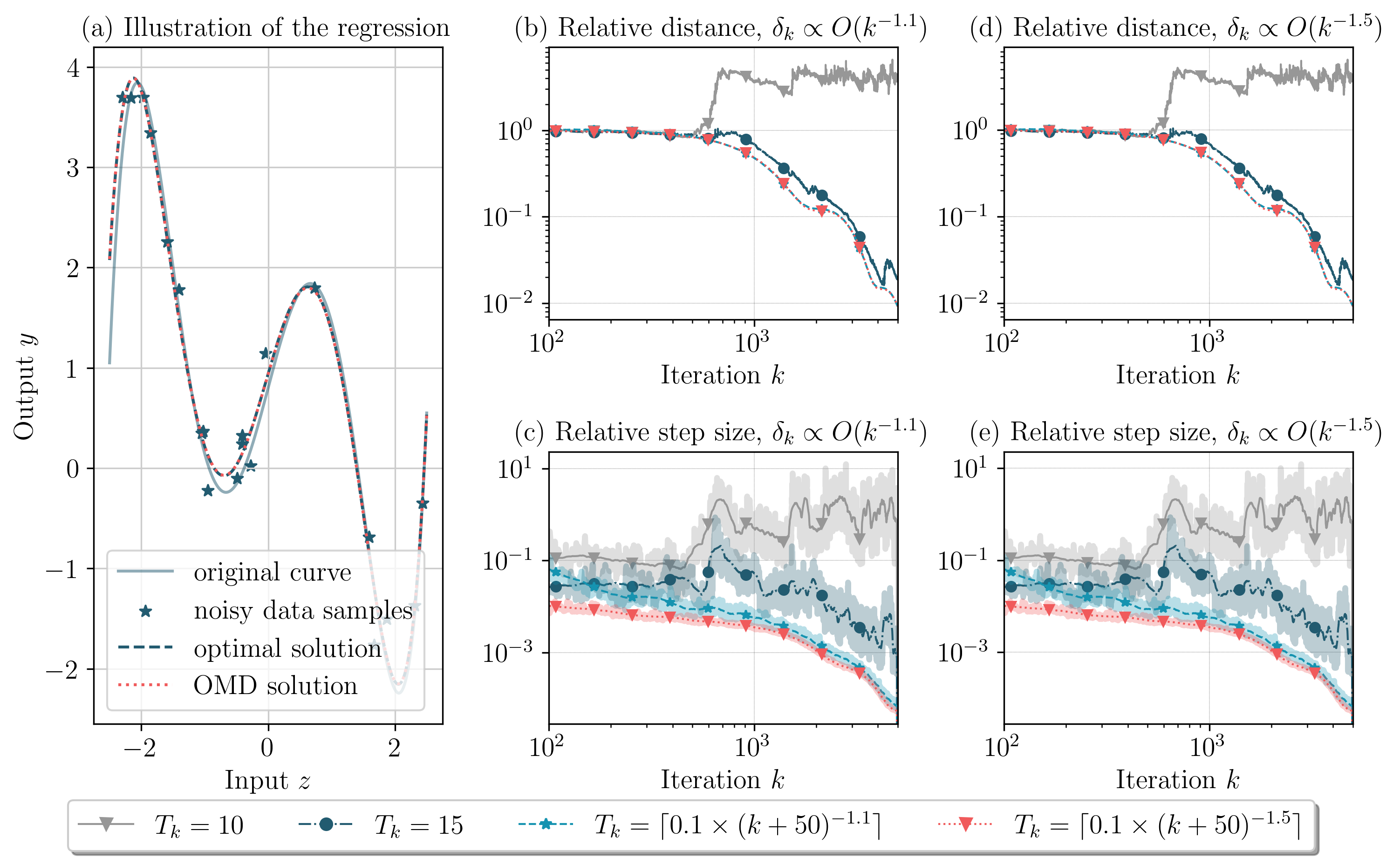}
    \caption{Performance of Algorithm~\ref{alg:vr-bdt-omd} in the Least Square Estimation: \Tblue{in Panel (a), the optimal solution is obtained by solving \eqref{eq:lse-obj} analytically; the OMD solution corresponds to the case when $\delta_k = O(k^{-1.1})$ and $T_k = \ceil{0.1 \times (k+50)^{-1.1}}$}; Panel (b) and (d) visualize the relative distance to the unique CP, i.e., the metric is given by $\norm{\hat{X}_{k+1/2} - x_*}_2/\norm{x_*}_2$; Panel (c) and (e) report the relative updating step sizes per iteration, i.e., $\norm{\hat{X}_{k+1/2} - \hat{X}_{k-1/2}}_2/\norm{\hat{X}_{k-1/2}}_2$. The rolling averages with a window size of 100 and the original fluctuations are illustrated in solid curves and semi-transparent curves, respectively. }
    \label{fig:lse}
\end{figure}

\section{CONCLUSION}
In this work, we investigate bandit learning in multi-player continuous games with an emphasis on handling merely coherent cases.  
A new learning algorithm is proposed, by integrating the idea of optimistic mirror descent and multi-point pseudo-gradient estimation. 
Under the assumptions posited and the conditions that the sequences of query radius $\delta_k$ and the reciprocal of sample size $T_k$ are absolutely summable, the actual sequence of play generated by the proposed algorithm is shown to converge a.s. to a CP of the game. 
There are several potential directions for future exploration. 
The first one is relaxing the requirements for the number of samples per iteration $T_k$, since the superlinear growth of $T_k$ may prevent the application of the proposed algorithm when the bandit feedback is inadequate. 
Furthermore, when it comes to a large-scale player network, the asynchronicity of the updates is a prevalent issue and the cost of synchronization is prohibitive, which is further exacerbated by the multi-point scheme considered. 
We intend to address these questions in future work.

\ifarxiv
\appendices 
\input{appendix}
\fi

\bibliographystyle{IEEEtran}
\bibliography{IEEEabrv,references}

\end{document}

%% file: macros.tex
\newcommand{\abs}[1]{\lvert#1\rvert}%
\newcommand{\norm}[1]{\lVert#1\rVert}

\newcommand{\Bnorm}[1]{\Big\lVert#1\Big\rVert}

\newcommand{\cl}[1]{\text{cl}(#1)}
\newcommand{\expt}[2]{\mathbb{E}_{#1}[#2]}
\newcommand{\Bexpt}[2]{\mathbb{E}_{#1}\Big[#2\Big]}
\DeclarePairedDelimiter\ceil{\lceil}{\rceil} 

\DeclareMathOperator{\dom}{dom}

\DeclareMathOperator{\minimize}{minimize}

\newcommand{\playerN}{\mathcal{N}}

\newcommand{\bone}{\boldsymbol 1}
\newcommand{\bzero}{\boldsymbol 0}

\newcommand{\rset}[2]{\mathbb{R}^{#1}_{#2}}

\newcommand{\nset}[2]{\mathbb{N}^{#1}_{#2}}

\newcommand{\vol}[1]{\text{vol}(#1)}

\DeclareMathOperator{\argmin}{argmin}
\DeclareMathOperator{\argmax}{argmax}

\usepackage{accents}

\newcommand{\Tblue}[1]{\textcolor{black}{#1}}

%% file: appendix.tex
\section*{Appendix}
\addcontentsline{toc}{section}{Appendix}
\renewcommand{\thesubsection}{\Alph{subsection}}

\newtheorem{appdxlemma}{Lemma}
\numberwithin{appdxlemma}{subsection} 
\numberwithin{equation}{subsection} 

\subsection{Proof of Lemma~\ref{le:bias}}\label{pf:bias}

By the tower property $\Tilde{\mathcal{F}}_k \coloneqq \sigma\{\mathcal{F}_k \cup \sigma\{u_{k,0}\}\} \supseteq \mathcal{F}_k$ and the linearity of conditional expectation, we have
\begin{align*}
& \expt{}{G^i_k \mid \mathcal{F}_k} = \Bexpt{}{\expt{}{G^i_k \mid \Tilde{\mathcal{F}}_k}\mid \mathcal{F}_k} \\
& = \frac{1}{T_k} \sum_{t=1}^{T_k}\Bexpt{}{\frac{n^i}{\delta_k}\expt{}{\Big(J^i(\hat{X}_{k+1/2,t}) - J^i(\hat{X}_{k+1/2,0})\Big)u^i_{k,t} \mid \Tilde{\mathcal{F}}_k} \mid \mathcal{F}_k}. 
\end{align*}
For every $t \in \{1, \ldots, T_k\}$, it follows from Lemma~1 of \cite{huang2023zeroth} that $\nabla_{x^i} \Tilde{J}^i_{\delta_k}(\bar{X}_{k+1/2})$ is a version of $\frac{n^i}{\delta_k}\expt{}{(J^i(\hat{X}_{k+1/2,t}) - J^i(\hat{X}_{k+1/2,0}))u^i_{k,t} \mid \Tilde{\mathcal{F}}_k}$. 
Based on the fact that $\bar{X}_{k+1/2} \in \mathcal{F}_k$, we have the following relation holds a.s.:
\begin{align*}
& \expt{}{G^i_k \mid \mathcal{F}_k} = \frac{1}{T_k} \sum_{t=1}^{T_k}\expt{}{\nabla_{x^i} \Tilde{J}^i_{\delta_k}(\bar{X}_{k+1/2}) \mid \mathcal{F}_k} = \nabla_{x^i} \Tilde{J}^i_{\delta_k}(\bar{X}_{k+1/2}). 
\end{align*}
With the above results in hand, the norm of systematic error $\norm{B^i_k}$ can be reformulated as $\norm{B^i_k} = \norm{\nabla_{x^i}\tilde{J}^i_{\delta_k}(\bar{X}_{k+1/2}) - \nabla_{x^i}J^i(X_{k+1/2})}$, and the proof for Lemma~2 of \cite{huang2023zeroth} directly carries over.

\subsection{Proof of Lemma~\ref{le:variance}}\label{pf:variance}
Using the definition of \eqref{eq:mpg} and the linearity of conditional expectation, we have:
\begin{align*}
& \expt{}{\norm{G^i_k}^2 \mid \Tilde{\mathcal{F}}_k} = \Big(\frac{n^i}{\delta_k T_k}\Big)^2 \Bexpt{}{\Bnorm{\sum_{t=1}^{T_k} (\hat{J}^i_{k,t} - \hat{J}^i_{k,0})u^i_{k,t}}^2 \mid \Tilde{\mathcal{F}}_k} \\
& = \Big(\frac{n^i}{\delta_k T_k}\Big)^2 \Big(\sum_{t=1}^{T_k}\expt{}{\norm{ (\hat{J}^i_{k,t} - \hat{J}^i_{k,0})u^i_{k,t}}^2 \mid \Tilde{\mathcal{F}}_k} + \sum_{1 \leq s, t \leq T_k, s\neq t} \\
& \expt{}{(\hat{J}^i_{k,s} - \hat{J}^i_{k,0})(\hat{J}^i_{k,t} - \hat{J}^i_{k,0}) \cdot \langle u^i_{k,s}, u^i_{k,t}\rangle \mid \Tilde{\mathcal{F}}_k}\Big). 
\end{align*}
For each pair $(s, t)$ with $s \neq t$, denote $\Tilde{\mathcal{F}}_{k,s} \coloneqq \sigma\{\Tilde{\mathcal{F}}_k\cup \sigma\{u_{k,s}\}\}$ and the conditional expectation of the inner product can be reformulated as follows:
\begin{align*}
& \Big(\frac{n^i}{\delta_k}\Big)^2\expt{}{\langle (\hat{J}^i_{k,s} - \hat{J}^i_{k,0})u^i_{k,s}, (\hat{J}^i_{k,t} - \hat{J}^i_{k,0})u^i_{k,t}\rangle \mid \Tilde{\mathcal{F}}_k} \\
& = \Bexpt{}{\expt{}{\langle \frac{n^i}{\delta_k}(\hat{J}^i_{k,s} - \hat{J}^i_{k,0})u^i_{k,s}, \frac{n^i}{\delta_k}(\hat{J}^i_{k,t} - \hat{J}^i_{k,0})u^i_{k,t}\rangle \mid \Tilde{\mathcal{F}}_{k,s}} \mid \Tilde{\mathcal{F}}_k} \\
& = \Bexpt{}{\langle \frac{n^i}{\delta_k}(\hat{J}^i_{k,s} - \hat{J}^i_{k,0})u^i_{k,s}, \expt{}{\frac{n^i}{\delta_k}(\hat{J}^i_{k,t} - \hat{J}^i_{k,0})u^i_{k,t}\mid \Tilde{\mathcal{F}}_{k,s}}\rangle\mid \Tilde{\mathcal{F}}_k} \\
& = \Bexpt{}{\langle \frac{n^i}{\delta_k}(\hat{J}^i_{k,s} - \hat{J}^i_{k,0})u^i_{k,s}, \nabla_{x^i} \Tilde{J}^i_{\delta_k}(\bar{X}_{k+1/2})\rangle\mid \Tilde{\mathcal{F}}_k} \\
& = \langle \expt{}{\frac{n^i}{\delta_k}(\hat{J}^i_{k,s} - \hat{J}^i_{k,0})u^i_{k,s}\mid \Tilde{\mathcal{F}}_k}, \nabla_{x^i} \Tilde{J}^i_{\delta_k}(\bar{X}_{k+1/2})\rangle \\
& = \langle \nabla_{x^i} \Tilde{J}^i_{\delta_k}(\bar{X}_{k+1/2}), \nabla_{x^i} \Tilde{J}^i_{\delta_k}(\bar{X}_{k+1/2})\rangle = \norm{\nabla_{x^i} \Tilde{J}^i_{\delta_k}(\bar{X}_{k+1/2})}^2 \;\text{a.s.}
\end{align*}
Combining the observations above yields: 
\begin{align*}
& \expt{}{\norm{G^i_k}^2 \mid \Tilde{\mathcal{F}}_k} = \Big(\frac{n^i}{\delta_k T_k}\Big)^2 \sum_{t=1}^{T_k}\expt{}{\norm{ (\hat{J}^i_{k,t} - \hat{J}^i_{k,0})u^i_{k,t}}^2 \mid \Tilde{\mathcal{F}}_k} \\
& \qquad + (1 - \frac{1}{T_k})\norm{\nabla_{x^i} \Tilde{J}^i_{\delta_k}(\bar{X}_{k+1/2})}^2, \; \text{a.s.}
\end{align*}
For the stochastic error $V^i_k \coloneqq G^i_k - \expt{}{G^i_k \mid \mathcal{F}_k} = G^i_k - \nabla_{x^i} \Tilde{J}^i_{\delta_k}(\bar{X}_{k+1/2})$, applying the results for $\expt{}{\norm{G^i_k}^2 \mid \Tilde{\mathcal{F}}_k}$ gives:
\begin{align*}
& \expt{}{\norm{V^i_k}^2 \mid \Tilde{\mathcal{F}}_k} = \expt{}{\norm{G^i_k - \nabla_{x^i} \Tilde{J}^i_{\delta_k}(\bar{X}_{k+1/2})}^2 \mid \Tilde{\mathcal{F}}_k} \\
& = \expt{}{\norm{G^i_k}^2 \mid \Tilde{\mathcal{F}}_k} - 2\expt{}{\langle G^i_k, \nabla_{x^i} \Tilde{J}^i_{\delta_k}(\bar{X}_{k+1/2})\rangle \mid \Tilde{\mathcal{F}}_k} \\
& \qquad + \norm{\nabla_{x^i} \Tilde{J}^i_{\delta_k}(\bar{X}_{k+1/2})}^2 \\
& = \expt{}{\norm{G^i_k}^2 \mid \Tilde{\mathcal{F}}_k} - \norm{\nabla_{x^i} \Tilde{J}^i_{\delta_k}(\bar{X}_{k+1/2})}^2 \\
& = \Big(\frac{n^i}{\delta_k T_k}\Big)^2 \sum_{t=1}^{T_k}\expt{}{\norm{ (\hat{J}^i_{k,t} - \hat{J}^i_{k,0})u^i_{k,t}}^2 \mid \Tilde{\mathcal{F}}_k} \\
& \qquad - \frac{1}{T_k}\norm{\nabla_{x^i} \Tilde{J}^i_{\delta_k}(\bar{X}_{k+1/2})}^2 \\
& \leq \Big(\frac{n^i}{\delta_k}\Big)^2\frac{1}{T_k}\expt{}{(\hat{J}^i_{k,t} - \hat{J}^i_{k,0})^2\norm{ u^i_{k,t}}^2 \mid \Tilde{\mathcal{F}}_k} \;\text{a.s.}
\end{align*}
The difference $(\hat{J}^i_{k,t} - \hat{J}^i_{k,0})^2$ can be further upper bounded as: 
\begin{align}\label{eq:var-eq1}
\begin{split}
& (\hat{J}^i_{k,t} - \hat{J}^i_{k,0})^2 \overset{(a)}{=} (\langle \nabla_x J^i(Z), \hat{X}_{k+1/2,t} - \hat{X}_{k+1/2,0}\rangle)^2 \\
& \leq \norm{\nabla_x J^i(Z)}^2 \cdot \norm{\hat{X}_{k+1/2,t} - \hat{X}_{k+1/2,0}}^2 \\
& \overset{(b)}{\leq} \bar{\nabla}^2_i \cdot \delta_k^2 \norm{u_{k,t} - u_{k,0}}^2 = 4N\bar{\nabla}^2_i\delta_k^2, 
\end{split}
\end{align}
where, in $(a)$, we apply the mean value theorem for differentiable function and let $Z$ denote some convex combination of $\hat{X}_{k+1/2,t}$ and $\hat{X}_{k+1/2,0}$; 
for the relation $(b)$ we let $\bar{\nabla}_i \coloneqq \max_{x \in \mathcal{X}} \norm{\nabla_{x}J^i(x)}_2$ and apply the definition in \eqref{eq:perb}. 
Consequently, it can be directly inferred that: 
\begin{align*}
& \expt{}{\norm{V^i_k}^2 \mid \Tilde{\mathcal{F}}_k} \leq 4N(\bar{\nabla}_in^i)^2/T_k, \\
& \expt{}{\norm{V_k}^2 \mid \Tilde{\mathcal{F}}_k} \leq 4N\sum_{i\in \mathcal{N}}(\bar{\nabla}_in^i)^2/T_k. 
\end{align*}


\subsection{Proof of Theorem~\ref{thm:ext-rs}}\label{pf:ext-rs}
Before proceeding, we attribute the proving technique leveraged below to that of \cite[Thm.~2.3.5]{gadat2017stochastic}, while we provide complete proof for a simplified version and fill out some omitted steps of the reference for the completeness of this work. 
By letting $\hat{\zeta}_k \coloneqq \sum_{t=2}^k \zeta_t$ for $k \geq 2$ and $\hat{\zeta}_1 = 0$, the recurrent inequality can be expressed as
\begin{align}\label{eq:recur-1}
\expt{}{Z_{k+1} + \hat{\zeta}_{k+1} \mid \mathcal{F}_k} \leq  Z_{k} + \hat{\zeta}_{k} + \xi_k, \forall k \in \nset{}{+}. 
\end{align}
Likewise, let $\hat{\xi}_k \coloneqq \sum_{t=2}^k \xi_t$ for $k \geq 2$ and $\hat{\xi}_1 = 0$, and we have $0 \leq \hat{\xi}_k \nearrow \hat{\xi}_\infty$. 
It follows from the monotone convergence theorem that $\expt{}{\hat{\xi}_k} \nearrow \expt{}{\hat{\xi}_{\infty}}$ and $\sum_{k \in \nset{}{+}} \expt{}{\xi_k} < \infty$ implies $\expt{}{\hat{\xi}_{\infty}} < \infty$. 
Through the integration of this definition into \eqref{eq:recur-1}, we can construct a new recurrent inequality as follows: 
\begin{align}
\begin{split}
& \expt{}{Z_{k+1} + \hat{\zeta}_{k+1} + \expt{}{\hat{\xi}_{\infty} \mid \mathcal{F}_{k+1}} - \hat{\xi}_{k+1}\mid \mathcal{F}_k} \\
& \qquad \leq Z_{k} + \hat{\zeta}_{k} + \expt{}{\hat{\xi}_{\infty}\mid \mathcal{F}_k} - \hat{\xi}_{k}.
\end{split}
\end{align}
Based on the observation that $\hat{\xi}_{\infty} - \hat{\xi}_k \geq 0$, we can let $\Tilde{Z}_k \coloneqq Z_{k} + \hat{\zeta}_{k} + \expt{}{\hat{\xi}_{\infty}\mid \mathcal{F}_k} - \hat{\xi}_{k}$, which forms a sequence of non-negative random variables, and deduce that:
\begin{align}
\expt{}{\Tilde{Z}_{k+1} \mid \mathcal{F}_k} \leq \Tilde{Z}_k. 
\end{align}
Furthermore, for each $k \in \nset{}{+}$, $\expt{}{\Tilde{Z}_k} \leq \expt{}{\Tilde{Z}_1} = \expt{}{Z_1} + \expt{}{\hat{\xi}_{\infty}} < \infty$, which together with the preceding observations indicates that $(\Tilde{Z}_k)_{k \in \nset{}{+}}$ is a non-negative super-martingale. 
Straightforward application of the martingale convergence theorem yields: $\lim_{k \to \infty} \Tilde{Z}_k = \Tilde{Z}_\infty$ a.s. where $\Tilde{Z}_\infty$ is a $L^1$ random variable, i.e., $\expt{}{\abs{\Tilde{Z}_{\infty}}} < \infty$. 
Denote $\hat{\xi}^c_k \coloneqq \expt{}{\hat{\xi}_{\infty} \mid \mathcal{F}_k} - \hat{\xi}_k \in \mathcal{F}_k$. 
Note that $(\hat{\xi}^c_k)_{k \in \nset{}{+}}$ is a non-negative super-martingale and $\lim_{k \to \infty} \expt{}{\hat{\xi}^c_k} = \lim_{k \to \infty} \expt{}{\expt{}{\hat{\xi}_{\infty} \mid \mathcal{F}_k} - \hat{\xi}_k} = \lim_{k \to \infty}(\expt{}{\hat{\xi}_{\infty}} - \expt{}{\hat{\xi}_k}) = \expt{}{\hat{\xi}_{\infty}} - \lim_{k \to \infty}\expt{}{\hat{\xi}_k} = 0$ as demonstrated earlier, and thus $\hat{\xi}^c_k \overset{k \to \infty}{\to} 0$ a.s.. 
As a result, $\lim_{k\to\infty}(Z_{k} + \hat{\zeta}_k) = \Tilde{Z}_{\infty}$ a.s. 
Since the sequence $(\hat{\zeta}_k)_{k \in \nset{}{+}}$ is non-negative, monotonically increasing and bounded from above, its limit exists a.s., i.e., $\lim_{k\to\infty} \hat{\zeta}_k = \hat{\zeta}_{\infty}$ a.s. 
Moreover, due to the surrogate relation that $\hat{\zeta}_{\infty} \leq \Tilde{Z}_{\infty}$ and $\expt{}{\Tilde{Z}_{\infty}} < \infty$, we then obtain $\expt{}{\hat{\zeta}_{\infty}} < \infty$. 
Therefore, we arrive at the conclusion that $\sum_{k \in \nset{}{+}} \zeta_k = \lim_{k \to \infty}\hat{\zeta}_k < \infty$ a.s. and $\lim_{k \in\nset{}{+}} Z_k = \Tilde{Z}_{\infty} - \hat{\zeta}_{\infty}$ a.s. and the limit is $L^1$, i.e., $\expt{}{\Tilde{Z}_{\infty} - \hat{\zeta}_{\infty}} < \infty$.

\subsection{Proof of Theorem~\ref{thm:convg}}\label{pf:convg}
By applying the standing recurrent inequality of OMD \cite[Lem~A.2]{huang2023zeroth}\cite[Prop.~B.3]{mertikopoulos2018optimistic} and letting $x_*$ denote one CP of $\mathcal{G}$, we can obtain the following relation for the $k$-th iteration:
\begin{align*}
& D(x_*, X_{k+1}) \leq D(x_*, X_k) - \tau \langle G_k, X_{k+1/2} - x_*\rangle \\
& \qquad + \frac{\tau^2}{2\Tilde{\mu}}\norm{G_k - G_{k-1}}^2 - \frac{\Tilde{\mu}}{2}\norm{X_{k+1/2} - X_{k}}^2 \\
& \leq D(x_*, X_k) - \tau \langle F(X_{k+1/2}), X_{k+1/2} - x_*\rangle - \frac{\Tilde{\mu}}{2}\norm{X_{k+1/2} - X_{k}}^2 \\
& \qquad - \tau \langle B_k, X_{k+1/2} - x_*\rangle - \tau \langle V_k, X_{k+1/2} - x_*\rangle \\
& \qquad + \frac{\tau^2}{2\Tilde{\mu}}\norm{F(X_{k+1/2}) - F(X_{k-1/2}) + B_k - B_{k-1} + V_k - V_{k-1}}^2. 
\end{align*}
Since $x_*$ is a CP of $\mathcal{G}$, we have $\langle F(X_{k+1/2}), X_{k+1/2} - x_*\rangle \geq 0$, as posited in Assumption~\ref{asp:vs}. 
Now take the conditional expectation $\expt{}{\cdot \mid \mathcal{F}_k}$ of both sides of the above inequality. 
For the inner product of systematic error $B_k$, in light of Lemma~\ref{le:bias}, $-\langle B_k, X_{k+1/2} - x_*\rangle \leq \alpha_BD_{\mathcal{X}}\delta_k$, where $D_{\mathcal{X}}$ denotes the diameter of the feasible set, i.e., $D_{\mathcal{X}} \coloneqq \max_{x, y \in \mathcal{X}}\norm{x - y}$. 
Since $\expt{}{V_k \mid \mathcal{F}_k} = 0$ and $X_{k+1/2} \in \mathcal{F}_k$, $\expt{}{\langle V_k, X_{k+1/2} - x_*\rangle \mid \mathcal{F}_k} = \langle \expt{}{V_k\mid \mathcal{F}_k}, X_{k+1/2} - x_*\rangle = 0$. 
By appealing to the Cauchy-Schwarz inequality and the $L$-Lipschitz continuity of $F$, we can derive that 
\begin{align}\label{eq:convg-1}
\begin{split}
& \expt{}{D(x_*, X_{k+1}) \mid \mathcal{F}_k} \leq D(x_*, X_k) - \frac{\Tilde{\mu}}{2}\norm{X_{k+1/2} - X_{k}}^2 \\
&  \qquad + \frac{(\tau L)^2}{\Tilde{\mu}}\norm{X_{k+1/2} - X_{k-1/2}}^2 + \hat{\Delta}_{k,1}, 
\end{split}
\end{align}
where $\hat{\Delta}_{k,1} \coloneqq C_{\Delta, 1}\big(\delta_k + \delta_k^2 + \delta_{k-1}^2 + \expt{}{\norm{V_k}^2 \mid \mathcal{F}_k} + \norm{V_{k-1}}^2\big)$ with some fixed constant $C_{\Delta, 1}$. 

In order to facilitate the convergence analysis in the merely coherent scenario, we are led to upper bound $-\norm{X_{k+1/2} - X_k}^2$ as follows: 
\begin{align*}
& -\norm{X_{k+1/2} - X_k}^2 \leq -\frac{1}{2}\norm{X_k - P_{X_k, \mathcal{X}}(-\tau F(X_k))}^2 \\
& \qquad + \norm{P_{X_k, \mathcal{X}}(-\tau G_{k-1}) - P_{X_k, \mathcal{X}}(-\tau F(X_k))}^2 \\
& \leq -\frac{1}{2}\norm{X_k - P_{X_k, \mathcal{X}}(-\tau F(X_k))}^2 + \frac{\tau^2}{\Tilde{\mu}^2}\norm{G_{k-1} - F(X_k)}^2 \\
& \leq -\frac{1}{2}\norm{X_k - P_{X_k, \mathcal{X}}(-\tau F(X_k))}^2 + 2\big(\frac{\tau L}{\Tilde{\mu}}\big)^2\norm{X_{k-1/2} - X_k}^2 \\
& \qquad + 2\big(\frac{\tau}{\Tilde{\mu}}\big)^2\norm{B_{k-1} + V_{k-1}}^2 \\
& \leq -\frac{1}{2}\norm{X_k - P_{X_k, \mathcal{X}}(-\tau F(X_k))}^2 + 4\big(\frac{\tau L}{\Tilde{\mu}}\big)^2\norm{X_{k-1/2} - X_{k+1/2}}^2 \\
& \qquad + 4\big(\frac{\tau L}{\Tilde{\mu}}\big)^2\norm{X_{k+1/2} - X_{k}}^2 + 2\big(\frac{\tau}{\Tilde{\mu}}\big)^2\norm{B_{k-1} + V_{k-1}}^2, 
\end{align*}
where $\varepsilon(x) \coloneqq \norm{x - P_{x, \mathcal{X}}(-\tau F(x))}^2$ serves as a residual function. 
By the observation that $\varepsilon(x_\star) = 0$ is equivalent to the zero inclusion that $0 \in N_{\mathcal{X}}(x_\star) + \tau F(x_\star)$, we can assert that $x_{\star}$ is a CP of $\mathcal{G} \iff \varepsilon(x_\star) = 0$. 
In light of the upper bound derived above and the choice of step size $(\tau L/\Tilde{\mu})^2 \leq 1/12$, \eqref{eq:convg-1} can be reformulated as: 
\begin{align}\label{eq:convg-2}
\begin{split}
& \expt{}{D(x_*, X_{k+1}) \mid \mathcal{F}_k} \leq D(x_*, X_k) -\frac{\Tilde{\mu}}{2}(1 - \frac{1}{10})\norm{X_{k+1/2} - X_{k}}^2 \\
&  + \frac{1}{10}\big(-\frac{\Tilde{\mu}}{4}\varepsilon(X_k) + \frac{\Tilde{\mu}}{6}\norm{X_{k-1/2} - X_{k+1/2}}^2 + \frac{\Tilde{\mu}}{6}\norm{X_{k+1/2} - X_{k}}^2\big) \\
&   + \frac{\Tilde{\mu}}{12}\norm{X_{k+1/2} - X_{k-1/2}}^2 + \hat{\Delta}_{k,2}, 
\end{split}
\end{align}
where $\hat{\Delta}_{k,2} \coloneqq C_{\Delta, 2}\big(\delta_k + \delta_k^2 + \delta_{k-1}^2 + \expt{}{\norm{V_k}^2 \mid \mathcal{F}_k} + \norm{V_{k-1}}^2\big)$ with some larger fixed constant $C_{\Delta, 2}$.

We reapply the Cauchy-Schwarz inequality to $\norm{X_{k+1/2} - X_{k-1/2}}^2$, yielding 
\begin{align*}
\norm{X_{k+1/2} - X_{k-1/2}}^2 \leq 2\norm{X_{k+1/2} - X_{k}}^2 + 2\norm{X_{k} - X_{k-1/2}}^2,
\end{align*}
while it can be recursively obtained that for all $k \geq 3$, 
\begin{align*}
& \norm{X_{k} - X_{k-1/2}}^2 = \norm{\nabla \psi^*(\nabla \psi(X_{k-1}) - \tau G_{k-1}) \\
& - \nabla \psi^*(\nabla \psi(X_{k-1}) - \tau G_{k-2})}^2 \leq (\tau/\Tilde{\mu})^2\norm{G_{k-1} - G_{k-2}}^2 \\
& \leq 2(\tau L/\Tilde{\mu})^2\norm{X_{k-1/2} - X_{k-3/2}}^2 + 2(\tau/\Tilde{\mu})^2 \Delta_{k,3}, 
\end{align*}
with $ \Delta_{k,3} \coloneqq \norm{B_{k-1} - B_{k-2} + V_{k-1} - V_{k-2}}^2$. 

Adding $(\Tilde{\mu}/10)\norm{X_{k+1/2} - X_{k-1/2}}^2$ to both sides of \eqref{eq:convg-2} and substituting $\norm{X_{k+1/2} - X_{k-1/2}}^2$ of R.H.S. with the proceeding inequality produces:
\begin{align*}
\begin{split}
& \expt{}{D(x_*, X_{k+1}) + \frac{\Tilde{\mu}}{10}\norm{X_{k+1/2} - X_{k-1/2}}^2 \mid \mathcal{F}_k} \leq D(x_*, X_k) \\
&  -\frac{13\Tilde{\mu}}{30}\norm{X_{k+1/2} - X_{k}}^2 - \frac{\Tilde{\mu}}{40}\varepsilon(X_k) + \frac{\Tilde{\mu}}{5}\norm{X_{k+1/2} - X_{k-1/2}}^2 + \hat{\Delta}_{k,2} \\
& \leq D(x_*, X_k) + \frac{\Tilde{\mu}}{15}\norm{X_{k-1/2} - X_{k-3/2}}^2 \\
&  -\frac{\Tilde{\mu}}{30}\norm{X_{k+1/2} - X_{k}}^2 - \frac{\Tilde{\mu}}{40}\varepsilon(X_k) + \hat{\Delta}_{k,3}, \forall k \geq 3, \\
\end{split}
\end{align*}
where $\hat{\Delta}_{k,3} \coloneqq C_{\Delta,3}\big(\delta_k + \sum_{t=k-2}^k (\delta_t^2 + \expt{}{\norm{V_t}^2 \mid \mathcal{F}_k})\big)$ for some further larger constant $C_{\Delta,3}$. 
Further manipulating the coefficients of $\norm{X_{k+1/2} - X_{k-1/2}}^2$ gives $\forall k \geq 3$: 
\begin{align*}
\begin{split}
& \expt{}{D(x_*, X_{k+1}) + \frac{\Tilde{\mu}}{15}\norm{X_{k+1/2} - X_{k-1/2}}^2 \mid \mathcal{F}_k} \\
& \leq D(x_*, X_k) + \frac{\Tilde{\mu}}{15}\norm{X_{k-1/2} - X_{k-3/2}}^2 - \frac{\Tilde{\mu}}{30}\norm{X_{k+1/2} - X_{k-1/2}}^2  \\
& - \frac{\Tilde{\mu}}{30}\norm{X_{k+1/2} - X_{k}}^2 - \frac{\Tilde{\mu}}{40}\varepsilon(X_k) + \hat{\Delta}_{k,3}. 
\end{split}
\end{align*}

Using Lemma~\ref{le:variance}, we have $\expt{}{\hat{\Delta}_{k,3}} \leq C_{\Delta,3}\big(\delta_k + \sum_{t=k-2}^k (\delta_t^2 + \alpha_V/T_t)\big)$, and the summability conditions $\sum_{k \in \nset{}{+}} \delta_k < \infty$ and $\sum_{k \in \nset{}{+}} 1/T_k < \infty$ entail that $\sum_{k \geq 3} \expt{}{\hat{\Delta}_{k,3}} < \infty$. 
Then the application of Theorem~\ref{thm:ext-rs} allows us to assert the following: 
\begin{outline}[enumerate]
\1 $\sum_{k \geq 3} \Tilde{\mu}/40 \cdot \varepsilon(X_k) < \infty$ a.s.;
\1 $\sum_{k \geq 3} \Tilde{\mu}/30 \cdot \norm{X_{k+1/2} - X_{k-1/2}}^2 < \infty$ a.s.; 
\1 $\sum_{k \geq 3} \Tilde{\mu}/30 \cdot \norm{X_{k+1/2} - X_{k}}^2 < \infty$ a.s.; 
\1 $D(x_*, X_{k+1}) + \Tilde{\mu}/15\cdot \norm{X_{k+1/2} - X_{k-1/2}}^2$ converges a.s. to some $L^1$ random variable. 
\end{outline}
These results entail that there exists a sample set $\hat{\Omega} \subseteq \Omega$ and $\mathcal{P}(\hat{\Omega}) = 1$ such that for any $\omega \in \hat{\Omega}$, the above statements $(\romannum{1})-(\romannum{4})$ hold true for the deterministic sequences $(X_k(\omega))_{k \in \nset{}{+}}$ and $(X_{k+1/2}(\omega))_{k \in \nset{}{+}}$. 
Moreover, since $\big(X_k(\omega)\big)_{k \in \nset{}{}} \in \mathcal{X}$ and the map $P_{x, \mathcal{X}}(-\tau F(x))$ is continuous in $x$, there exists a subsequence $(k_m)_{m \in \nset{}{+}}$ such that $X_{k_m}(\omega) \overset{m\to\infty}{\to} x_{\star}$ and $\lim_{m \to \infty} \varepsilon\big(X_{k_m}(\omega)\big) = \varepsilon\big(x_{\star}\big) = 0$, i.e., $x_\star$ is a CP of $\mathcal{G}$. 
We can then substitute $x_\star$ for $x_*$ in $(\romannum{4})$.  Since $(\romannum{2})$ suggests that $\norm{X_{k+1/2} - X_{k-1/2}}^2(\omega) \overset{k \to \infty}{\to} 0$, we can assert that $D(x_\star, X_{k}(\omega))$ admits a finite limit. 
In conjunction with Assumption~\ref{asp:recip}, it follows that $D(x_\star, X_{k_m}(\omega)) \overset{m \to \infty}{\to} 0$ and hence $D(x_\star, X_{k}(\omega)) \overset{k \to \infty}{\to} 0$, i.e., the base states $\big(X_{k}(\omega)\big)_{k \in \nset{}{+}}$ converge to $x_\star$. 
Combining this result with $(\romannum{3})$ yields that the leading states $\big(X_{k+1/2}(\omega)\big)_{k \in \nset{}{+}}$ converge to $x_\star$, and the a.s. convergence of the actual sequence of play $\big(\hat{X}_{k+1/2,t}(\omega)\big)_{k \in \nset{}{+}}$ to $x_\star$ directly derives from \eqref{eq:perb} and $\delta_k \overset{k \to \infty}{\to} 0$. 